\documentclass[11pt]{amsart}

\usepackage{geometry}
\usepackage{amssymb}
\usepackage{mathtools}
\usepackage{xcolor}
\usepackage{tikz}
\usepackage[colorlinks=true,citecolor=red,linkcolor=blue,urlcolor=red]{hyperref}

\numberwithin{equation}{section}
\newtheorem{thm}{Theorem}[section]
\newtheorem{lem}[thm]{Lemma}
\newtheorem{prop}[thm]{Proposition}

\theoremstyle{definition}
\newtheorem{defn}[thm]{Definition}
\newtheorem{ques}[thm]{Question}
\newtheorem{rem}[thm]{Remark}
\newtheorem{setup}[thm]{Set-up}

\begin{document}

\title{Embedded constant mean curvature tori in hemispheres of Berger spheres}

\author{Jihao Liu}
\address{School of Mathematical Sciences, Peking University, No. 5 Yiheyuan Road, Haidian District, Beijing 100871, China}
\address{Beijing International Center for Mathematical Research, Peking University, No. 5 Yiheyuan Road, Haidian District, Beijing 100871, China}
\email{liujihao@math.pku.edu.cn}

\author{Xiang Ma}
\address{Corresponding author. LMAM, School of Mathematical Sciences, Peking University, No. 5 Yiheyuan Road, Haidian District, Beijing 100871, China}
\email{maxiang@math.pku.edu.cn}

\date{\today}

\begin{abstract}
We prove that for all real numbers $\kappa,\tau$ with $\kappa>8\tau^2$ and $\tau\neq 0$, the Berger sphere $\mathbb{E}^3(\kappa,\tau)$ contains smooth closed embedded rotational surfaces of genus one with constant mean curvature which are contained in an open hemisphere. This disproves a conjecture of Fern\'andez and Mira. The main result of this paper is obtained by generative AI, particularly GPT-5.6-sol, Fable 5, and the Danus system, and then verified by the authors.
\end{abstract}

\maketitle

\tableofcontents

\section{Introduction}\label{sec:introduction}

The Alexandrov theorem is one of the fundamental results of the theory of constant mean curvature (CMC) surfaces: every compact embedded CMC surface in $\mathbb{R}^3$, in $\mathbb{H}^3$, or in an open hemisphere of the round sphere $\mathbb{S}^3$ is a round sphere \cite{Ale56}. The hemisphere hypothesis in the spherical case cannot be removed: the product tori $\mathbb{S}^1(r)\times\mathbb{S}^1(\sqrt{1-r^2})\subset\mathbb{S}^3$ are compact embedded CMC surfaces of genus one. Motivated by this statement, the problem of classifying all compact embedded CMC surfaces in a Riemannian $3$-manifold is called the \emph{Alexandrov problem} in that manifold.

Beyond the space forms, the most studied ambient spaces are the simply connected homogeneous $3$-manifolds $\mathbb{E}^3(\kappa,\tau)$ with $4$-dimensional isometry group, which fiber over the space form $\mathbb{M}^2(\kappa)$ with bundle curvature $\tau$; see \cite{Dan07,DHM09,MP12}. For $\tau=0$ these are the product spaces, where the Alexandrov problem was solved by Hsiang and Hsiang \cite{HH89} using horizontal Alexandrov reflection: every compact embedded CMC surface in $\mathbb{H}^2\times\mathbb{R}$, or in $\mathbb{S}^2_+\times\mathbb{R}$, is a rotational CMC sphere. For $\tau\neq 0$ the spaces $\mathbb{E}^3(\kappa,\tau)$ are the Berger spheres ($\kappa>0$), the Heisenberg group $\mathrm{Nil}_3$ ($\kappa=0$), and the universal cover of $\mathrm{PSL}_2(\mathbb{R})$ ($\kappa<0$). In these spaces the Hopf problem is solved: the holomorphic quadratic differential introduced by Abresch and Rosenberg for the product spaces \cite{AbRo04} and extended to all $\mathbb{E}^3(\kappa,\tau)$ in \cite{AbRo05} vanishes on every CMC sphere, and its vanishing forces rotational invariance, so every immersed CMC sphere in $\mathbb{E}^3(\kappa,\tau)$ is one of the \emph{canonical rotational CMC spheres}; see \cite[Section~4.3]{FM10}. The Alexandrov problem, however, has remained open for all three families, the essential difficulty being that the spaces $\mathbb{E}^3(\kappa,\tau)$ with $\tau\neq 0$ admit no reflections, so the Alexandrov reflection technique is unavailable; see \cite[Section~4.2]{FM10}.

\medskip

\noindent\textbf{The Alexandrov conjecture in Berger hemispheres.} In their 2010 ICM address, Fern\'andez and Mira singled out the Alexandrov problem for $\tau\neq 0$ as one of the major unsolved problems of the theory and recorded the expected answer as a conjecture: the canonical rotational CMC spheres are the only compact embedded CMC surfaces in $\mathrm{Nil}_3$, in the universal cover of $\mathrm{PSL}_2(\mathbb{R})$, and in hemispheres of Berger spheres \cite[Section~4.5]{FM10}.

The hemisphere restriction in the Berger case is forced by the Hopf tori. The Hopf lift of a circle in $\mathbb{M}^2(\kappa)$ is a flat embedded CMC torus in the Berger sphere (see \cite{Pin85} for the round case), and these tori realize every real value of the mean curvature (Proposition~\ref{prop:hopf-tori}). A Hopf torus contains complete Hopf fibers, and a complete Hopf fiber contains antipodal pairs of points of the underlying unit sphere $\mathbb{S}^3\subset\mathbb{C}^2$, so no Hopf torus is contained in an open hemisphere (Proposition~\ref{prop:hopf-tori}). The hemisphere clause therefore removes every previously known non-spherical compact embedded CMC surface in a Berger sphere, exactly as it removes the product tori in the round case.

The main result of this paper is that the Berger case of the conjecture fails, for every Berger sphere in the strongly squashed range $\kappa>8\tau^2$.

\begin{thm}[Embedded CMC tori in hemispheres]\label{thm:main-intro}
Let $\kappa,\tau$ be real numbers with $\tau\neq 0$ and $\kappa>8\tau^2$. Then there exists $H_*>0$ such that for every $H\geq H_*$, the Berger sphere $\mathbb{E}^3(\kappa,\tau)$ contains a smooth closed embedded rotational surface of genus one with constant mean curvature $H$ which is contained in an open hemisphere.

In particular, for every such $(\kappa,\tau)$ there exist compact embedded CMC surfaces contained in open hemispheres of $\mathbb{E}^3(\kappa,\tau)$ which are not spheres, and the Berger case of the conjecture of Fern\'andez and Mira fails for these parameters.
\end{thm}

The examples are completely explicit at the level of their defining data, and in one concrete case the closing condition admits a certificate by exact rational arithmetic.

\begin{thm}[An explicit torus]\label{thm:explicit-intro}
The Berger sphere $\mathbb{E}^3(4,\tfrac12)$, that is, the sphere $\mathbb{S}^3=\{(z,w)\in\mathbb{C}^2:|z|^2+|w|^2=1\}$ with the metric
\[
g=dr^2+\cos^2 r\,d\alpha^2+\sin^2 r\,d\beta^2-\tfrac{3}{4}\left(\cos^2 r\,d\alpha+\sin^2 r\,d\beta\right)^2
\]
in Hopf coordinates $(z,w)=(\cos r\,e^{i\alpha},\sin r\,e^{i\beta})$, contains a smooth closed embedded rotational surface of genus one with constant mean curvature $20$ which is contained in the open hemisphere $\{(z,w)\in\mathbb{S}^3:\operatorname{Re} z>0\}$.
\end{thm}

\medskip

\noindent\textbf{Strategy.} A surface in a Berger sphere which is invariant under the rotation $(z,w)\mapsto(z,e^{i\theta}w)$ is governed by a profile curve in the orbit space, and the CMC condition is equivalent to the profile being a magnetic geodesic of an explicit metric on the orbit space (Proposition~\ref{prop:profile-system}). The rotational symmetry produces a conserved quantity $e$. For $H>0$, we choose the energy regime $e\in(-H,0)$ and write $e=-aH$ with $a\in(0,1)$; then the radial coordinate $t=\sin^2 r$ oscillates between the two roots $t_1<a<t_2$ of the turning polynomial
\[
P(t)=t(1-t)-H^2(t-a)^2.
\]
The angular coordinate $\alpha$ advances by an explicit holonomy integral $\Delta\alpha(a)$ during each radial oscillation. The profile closes after a single oscillation, with zero winding around the rotation axis, exactly when $\Delta\alpha(a)=0$; if moreover its angular amplitude $A$ is smaller than $\pi/2$, the closed profile is a small lens whose angular coordinate stays within $A$ of the meridian $\alpha=0$, and its rotational surface is an embedded torus (Proposition~\ref{prop:lens}).

The heart of the construction is a sign change of the holonomy. We prove the uniform limit
\[
\lim_{H\to\infty}H^2\,\Delta\alpha(a)=\frac{\pi}{2\tau}\,\frac{d}{da}\sqrt{a\,(1-a+\tau^2 a)}
\]
in the normalized family $\kappa=4$, $\tau\in(0,\tfrac{1}{\sqrt2})$ (Proposition~\ref{prop:asymptotic}). The right-hand side changes sign at $a_c=\frac{1}{2(1-\tau^2)}$, and $a_c<1$ exactly when $\tau^2<\frac12$, which after undoing the normalization is the hypothesis $\kappa>8\tau^2$ of Theorem~\ref{thm:main-intro}. The intermediate value theorem then produces zero-holonomy parameters $a_H$ for all large $H$, the angular amplitude decays like $1/H$, and the resulting lens lifts to an embedded CMC torus inside the open hemisphere $\{\operatorname{Re}z>0\}$. Theorem~\ref{thm:explicit-intro} is obtained by certifying one such sign change at $\tau=\frac12$, $H=20$ by exact rational estimates.

\medskip

\noindent\textbf{Relation to the classifications of rotational CMC surfaces.} Rotationally invariant CMC surfaces in the Berger spheres and in the special linear group were classified by Torralbo \cite{Tor10}, and invariant CMC surfaces in the spaces $\mathbb{E}^3(\kappa,\tau)$ and in related spaces have since been studied extensively; see \cite{Tor12,TU12,Bue21,Bue22,KT24} and the surveys \cite{DHM09,MP12}. The compact embedded examples exhibited in \cite{Tor10} in Berger spheres close after winding around the rotation axis, the closing condition being that the angular period is of the form $2\pi/k$ with $k$ a positive integer; the tori of Theorems~\ref{thm:main-intro} and~\ref{thm:explicit-intro} instead have angular period equal to zero and do not wind around the axis. Along the tori constructed here the angular velocity changes sign, so they are of nodoid type in the terminology of \cite{Tor10,Bue22}. In the normalization $\kappa=4$, the period $T(H,E)$ of \cite[Section~3]{Tor10} agrees with our $\Delta\alpha$, with $E=e$. Although $T=0$ is formally allowed by the compactness condition that $T$ be a rational multiple of $\pi$, the zero-period embedded subcase was not isolated in the previous analyses: the nodoid discussion in \cite{Tor10} says only that the resulting surface is immersed, while item~(4) of \cite[Theorem~1.2]{Bue22} describes rotational nodoids of prescribed mean curvature in Berger spheres as properly immersed with self-intersections. Theorem~\ref{thm:explicit-intro} proves that at zero angular period the profile can instead be a simple embedded lens; we discuss this in Remark~\ref{rem:nodoid-correction}.

We also note that for \emph{non-constant} rotationally symmetric prescribed mean curvature, Bueno constructed embedded rotational tori in $\mathrm{Nil}_3$ and in the universal cover of $\mathrm{PSL}_2(\mathbb{R})$, so the prescribed-curvature analogue of Alexandrov uniqueness was already known to fail in those spaces \cite{Bue20,Bue21}. The constant mean curvature statement, which is the case recorded in the conjecture of \cite{FM10}, is the one settled (negatively, for Berger hemispheres) in this paper.

\medskip

\noindent\textbf{Scope.} Theorem~\ref{thm:main-intro} covers the entire parameter range $\kappa>8\tau^2$, and the construction produces embedded tori for every sufficiently large mean curvature in each such Berger sphere. The tools developed here extend further. The proof of the limit formula in Proposition~\ref{prop:asymptotic} extends verbatim to the whole squashed range $\tau\in(0,1)$, and it would be interesting to determine the exact set of parameters $(\kappa,\tau,H)$ for which zero-holonomy embedded lenses exist, in particular whether they reach the range $4\tau^2<\kappa\leq 8\tau^2$, the stretched Berger spheres $\kappa<4\tau^2$, or small values of $H$. In $\mathrm{Nil}_3$ and in the universal cover of $\mathrm{PSL}_2(\mathbb{R})$ there exist no compact rotational CMC tori, since rotational CMC tori in $\mathbb{E}^3(\kappa,\tau)$ exist only when $\kappa-4\tau^2>0$ (see \cite[Section~4.1]{FM10}), so the corresponding cases of the conjecture of \cite{FM10} would have to be attacked with non-rotational competitors. We record the natural remaining question for Berger spheres.

\begin{ques}\label{ques:all-berger}
Does every Berger sphere $\mathbb{E}^3(\kappa,\tau)$ with $\kappa\neq 4\tau^2$ contain a compact embedded CMC surface, other than a canonical rotational CMC sphere, which is contained in an open hemisphere?
\end{ques}

\begin{rem}\label{rem:ai}
The main result of this paper was obtained using generative AI, particularly GPT-5.6-sol, Fable 5, and the Danus system. Danus is a specialized agent built on the Rethlas system and substantially more capable of conducting fundamental mathematical research. Human verification and polishing were done afterwards. See \cite{Liu+26} and \cite{Ju+26} for detailed introductions to the Danus system and the Rethlas system, respectively. Due to the limitation of generative AI, it is possible that we have missed some related references in the literature, and we welcome any comments from experts.
\end{rem}

\subsection*{Acknowledgements}
The work was partially supported by the National Key R\&D Program of China \#\allowbreak 2024YFA1014400.
The first author would like to thank other members of the Danus team (namely Guoxiong Gao, Zeming Sun, Bin Wu, Shurui Liu, Jiedong Jiang, Haocheng Ju, Leheng Chen, Ronnie Cheng, Xiping Zhang, and Bin Dong) and the Rethlas team (namely Haocheng Ju, Jiedong Jiang, Shurui Liu, Guoxiong Gao, Yuefeng Wang, Zeming Sun, Bin Wu, Liang Xiao, and Bin Dong) for their contributions to the development of Danus and Rethlas.
The first author would like to thank Ruochuan Liu and Gang Tian for constant support and encouragement. The second author is partially supported by NSFC grant 11831005.

\section{Preliminaries}\label{sec:preliminaries}

In this section, we fix the model of the Berger spheres, prove the homothety normalization, and establish two hemisphere containment criteria together with the basic properties of the Hopf tori.

\begin{setup}\label{setup:berger}
Let $\mathbb{S}^3=\{(z,w)\in\mathbb{C}^2:|z|^2+|w|^2=1\}$, viewed as the unit sphere of the Euclidean space $\mathbb{R}^4=\mathbb{C}^2$. We use Hopf coordinates
\[
(z,w)=(\cos r\,e^{i\alpha},\sin r\,e^{i\beta}),\qquad r\in[0,\tfrac{\pi}{2}],\ \alpha,\beta\in\mathbb{R}/2\pi\mathbb{Z},
\]
and we write $s=\sin r$, $c=\cos r$, and $t=s^2$ throughout. For real numbers $\kappa>0$ and $\tau\neq 0$, the \emph{Berger sphere} $\mathbb{E}^3(\kappa,\tau)$ is $\mathbb{S}^3$ equipped with the metric
\begin{equation}\label{eq:berger-metric}
g_{\kappa,\tau}=\frac{4}{\kappa}\left[dr^2+c^2\,d\alpha^2+s^2\,d\beta^2+\Big(\frac{4\tau^2}{\kappa}-1\Big)\big(c^2\,d\alpha+s^2\,d\beta\big)^2\right].
\end{equation}
The metric $g_{\kappa,\tau}$ is round if and only if $\kappa=4\tau^2$. For $\tau>0$ we write
\begin{equation}\label{eq:normalized-metric}
g_\tau:=g_{4,\tau}=dr^2+c^2\,d\alpha^2+s^2\,d\beta^2+(\tau^2-1)\big(c^2\,d\alpha+s^2\,d\beta\big)^2
\end{equation}
for the normalized metric, and we set
\[
D=D_\tau(t)=c^2+\tau^2 s^2=1-t+\tau^2 t>0 .
\]
For an oriented surface with unit normal $N$, the second fundamental form and the mean curvature are taken with the conventions
\[
\mathrm{II}(X,Y)=g(\nabla_XY,N),\qquad H=\tfrac12\operatorname{trace}(\mathrm{II}).
\]
\end{setup}

The metric \eqref{eq:berger-metric} depends on $\tau$ only through $\tau^2$, so we lose nothing by assuming $\tau>0$. The following homothety lemma reduces all our statements to the normalized family \eqref{eq:normalized-metric}.

\begin{lem}[Homothety normalization]\label{lem:scaling}
Let $\kappa>0$, $\tau\neq 0$, and $\lambda>0$. Then
\[
g_{\lambda^2\kappa,\lambda\tau}=\lambda^{-2}g_{\kappa,\tau}.
\]
Moreover, if an oriented immersed surface in $\mathbb{S}^3$ has mean curvature $H$ with respect to $g_{\kappa,\tau}$ and unit normal $N$, then the same immersion has mean curvature $\lambda H$ with respect to $g_{\lambda^2\kappa,\lambda\tau}$ and unit normal $\lambda N$. In particular, taking $\lambda=2/\sqrt{\kappa}$ identifies $\big(\mathbb{E}^3(\kappa,\tau),H\big)$ with $\big((\mathbb{S}^3,g_{\tau'}),\tfrac{2}{\sqrt{\kappa}}H\big)$, where $\tau'=2|\tau|/\sqrt{\kappa}$, and
\[
0<\tau'<\tfrac{1}{\sqrt{2}}
\quad\Longleftrightarrow\quad
\kappa>8\tau^2 .
\]
\end{lem}

\begin{proof}
Substituting $(\lambda^2\kappa,\lambda\tau)$ into \eqref{eq:berger-metric} leaves the anisotropy coefficient unchanged, since $4(\lambda\tau)^2/(\lambda^2\kappa)=4\tau^2/\kappa$, and multiplies the prefactor by $\lambda^{-2}$. This proves the first identity.

Write $g=g_{\kappa,\tau}$ and $g'=\lambda^{-2}g$. Since $\lambda^{-2}$ is constant, the Levi-Civita connections of $g$ and $g'$ coincide. If $N$ is a $g$-unit normal of the surface, then $\lambda N$ is a $g'$-unit normal, and for tangent vectors $X,Y$,
\[
\mathrm{II}'(X,Y)=g'(\nabla_XY,\lambda N)=\lambda^{-1}\,\mathrm{II}(X,Y).
\]
If $(E_1,E_2)$ is a $g$-orthonormal tangent frame, then $(\lambda E_1,\lambda E_2)$ is a $g'$-orthonormal tangent frame, and
\[
H'=\tfrac12\sum_{j=1}^2 \mathrm{II}'(\lambda E_j,\lambda E_j)=\tfrac12\sum_{j=1}^2\lambda\,\mathrm{II}(E_j,E_j)=\lambda H .
\]
For $\lambda=2/\sqrt{\kappa}$ we obtain $\lambda^2\kappa=4$ and $\lambda|\tau|=2|\tau|/\sqrt{\kappa}=\tau'$. Finally, $(\tau')^2<\tfrac12$ means $4\tau^2/\kappa<\tfrac12$, which is equivalent to $\kappa>8\tau^2$ because $\kappa>0$.
\end{proof}

The phrase ``Berger hemispheres'' in \cite[Section~4.5]{FM10} is not accompanied there by a separate definition. Throughout this paper, we use the interpretation inherited from the underlying standard sphere, as follows.

\begin{defn}[Hemispheres]\label{defn:hemisphere}
For a Euclidean unit vector $v\in\mathbb{R}^4=\mathbb{C}^2$, the \emph{open hemisphere} with pole $v$ is
\[
\mathcal{H}_v=\{x\in\mathbb{S}^3: x\cdot v>0\},
\]
where $\cdot$ is the Euclidean inner product of $\mathbb{R}^4$. Equivalently, $\mathcal{H}_v$ is one of the two components of the complement of the great sphere $\mathbb{S}^3\cap v^{\perp}$. This notion depends only on the underlying sphere $\mathbb{S}^3$ and not on the choice of Berger metric on it.
\end{defn}

\begin{lem}[Containment criterion]\label{lem:hemisphere-criterion}
Let $\Sigma\subset\mathbb{S}^3$ be a subset such that every point of $\Sigma$ can be written as $(z,w)=(\sqrt{1-t}\,e^{i\alpha},\sqrt{t}\,e^{i\beta})$ with $t\leq t_+<1$ and with $\alpha$ admitting a real lift satisfying $|\alpha-\alpha_0|\leq\delta$, for some fixed $\alpha_0\in\mathbb{R}$ and $0\leq\delta<\tfrac{\pi}{2}$. Then $\Sigma\subset\mathcal{H}_v$ for $v=(e^{i\alpha_0},0)$, and every $x\in\Sigma$ satisfies
\[
x\cdot v\ \geq\ \sqrt{1-t_+}\,\cos\delta\ >\ 0 .
\]
\end{lem}

\begin{proof}
For $x=(\sqrt{1-t}\,e^{i\alpha},\sqrt{t}\,e^{i\beta})$ we compute
\[
x\cdot v=\operatorname{Re}\big(\sqrt{1-t}\,e^{i(\alpha-\alpha_0)}\big)=\sqrt{1-t}\,\cos(\alpha-\alpha_0).
\]
The hypotheses give $\sqrt{1-t}\geq\sqrt{1-t_+}$ and $\cos(\alpha-\alpha_0)\geq\cos\delta>0$.
\end{proof}

\begin{lem}[Invariant subsets and hemispheres]\label{lem:invariant-criterion}
Let $\Phi_\theta(z,w)=(z,e^{i\theta}w)$, and let $\Sigma\subset\mathbb{S}^3$ be a nonempty compact $\Phi_\theta$-invariant subset. Put $Z(\Sigma)=\{z\in\mathbb{C}:(z,w)\in\Sigma\text{ for some }w\}$. Then $\Sigma$ is contained in some open hemisphere if and only if there exists a nonzero $a\in\mathbb{C}$ with $\operatorname{Re}(z\bar a)>0$ for every $z\in Z(\Sigma)$. Consequently, no subset of $\mathbb{S}^3$ containing a complete Hopf fiber $\{(e^{i\theta}z_0,e^{i\theta}w_0):\theta\in\mathbb{R}\}$ is contained in any open hemisphere.
\end{lem}

\begin{proof}
Assume first that such an $a$ exists, and put $v=(a,0)/|a|$. For every $(z,w)\in\Sigma$ we have $(z,w)\cdot v=\operatorname{Re}(z\bar a)/|a|>0$, so $\Sigma\subset\mathcal{H}_v$.

Conversely, assume $\Sigma\subset\mathcal{H}_v$ for $v=(a,b)$ with $|v|=1$. Fix $(z,w)\in\Sigma$. By $\Phi_\theta$-invariance, $\operatorname{Re}(z\bar a)+\operatorname{Re}(e^{i\theta}w\bar b)>0$ for every $\theta\in\mathbb{R}$. Integrating this strict inequality of continuous functions over $\theta\in[0,2\pi]$ kills the second term and yields $2\pi\operatorname{Re}(z\bar a)>0$. Thus $\operatorname{Re}(z\bar a)>0$ for every $z\in Z(\Sigma)$, and in particular $a\neq 0$.

For the final assertion, a complete Hopf fiber contains the antipodal pair $\pm(z_0,w_0)$, attained at $\theta=0$ and $\theta=\pi$. If both lay in $\mathcal{H}_v$, then $x\cdot v>0$ and $-x\cdot v>0$ for $x=(z_0,w_0)$, which is impossible.
\end{proof}

We next record the Hopf tori, which realize every value of the mean curvature and explain the necessity of the hemisphere hypothesis in the conjecture of \cite{FM10}.

\begin{prop}[Hopf tori]\label{prop:hopf-tori}
Let $\kappa>0$, $\tau\neq 0$, and $r_0\in(0,\tfrac{\pi}{2})$, and let
\[
T_{r_0}=\{(z,w)\in\mathbb{S}^3:|z|=\cos r_0\}.
\]
Then $T_{r_0}$ is a closed embedded flat torus in $\mathbb{E}^3(\kappa,\tau)$ with constant mean curvature
\[
H=\frac{\sqrt{\kappa}}{2}\cot(2r_0)
\]
with respect to the unit normal $N=-\frac{\sqrt{\kappa}}{2}\,\partial_r$. In particular, every real number is the mean curvature of some Hopf torus. Moreover, $T_{r_0}$ contains complete Hopf fibers, and hence is contained in no open hemisphere.
\end{prop}

\begin{proof}
Let $M(r)$ denote the angular coefficient matrix of $g_{\kappa,\tau}$ on a level set of $r$. In the frame $(\partial_\alpha,\partial_\beta)$, it is
\[
M(r)=\frac{4}{\kappa}
\begin{pmatrix}
c^2+\eta c^4 & \eta c^2s^2\\
\eta c^2s^2 & s^2+\eta s^4
\end{pmatrix},
\qquad
\eta=\frac{4\tau^2}{\kappa}-1.
\]
Thus the induced metric on $T_{r_0}$ has coefficient matrix $M(r_0)$. A direct expansion gives
\[
\det M(r)=\frac{16}{\kappa^2}\,c^2s^2\big[1+\eta(c^2+s^2)\big]
=\frac{64\tau^2}{\kappa^3}\,c^2s^2>0,
\]
and
\[
M_{\alpha\alpha}(r)=\frac{4}{\kappa}c^2
\left(s^2+\frac{4\tau^2}{\kappa}c^2\right)>0.
\]
Thus $M(r_0)$ is positive definite, and the map $(\alpha,\beta)\mapsto(\cos r_0\,e^{i\alpha},\sin r_0\,e^{i\beta})$ is an injective immersion of a compact torus, hence an embedding with image $T_{r_0}$. Since the coefficients of $M(r_0)$ are constant in $(\alpha,\beta)$, all Christoffel symbols of the induced metric in these coordinates vanish, and $T_{r_0}$ is flat.

The ambient metric $g_{\kappa,\tau}$ satisfies $(g_{\kappa,\tau})_{rr}=\tfrac{4}{\kappa}$ and has no $dr$-angular cross terms, so $N=-\frac{\sqrt{\kappa}}{2}\partial_r$ is a unit normal along $T_{r_0}$. For angular coordinate fields $X,Y$, which commute with $\partial_r$ and whose inner products depend only on $r$, the Koszul formula gives
\[
2g_{\kappa,\tau}(\nabla_XY,\partial_r)
=-\partial_r\,g_{\kappa,\tau}(X,Y),
\]
and therefore
\[
\mathrm{II}(X,Y)=g_{\kappa,\tau}(\nabla_XY,N)
=\frac{\sqrt{\kappa}}{4}\,\partial_r\,g_{\kappa,\tau}(X,Y).
\]
Thus, as matrices,
\[
\mathrm{II}=\frac{\sqrt{\kappa}}{4}M'(r_0),
\]
where $'$ denotes $\partial_r$, and hence
\[
H=\tfrac12\operatorname{trace}\big(M(r_0)^{-1}\mathrm{II}\big)
=\left.\frac{\sqrt{\kappa}}{8}\,\partial_r\log\det M(r)\right|_{r=r_0}.
\]
The $r$-dependent factor of $\det M(r)$ is $c^2s^2$, and $\partial_r\log(c^2s^2)=2(\cot r-\tan r)=4\cot(2r)$. Evaluating at $r_0$ gives $H=\frac{\sqrt{\kappa}}{2}\cot(2r_0)$, and $\cot(2r_0)$ takes every real value as $r_0$ ranges over $(0,\tfrac{\pi}{2})$.

Finally, for $(z,w)\in T_{r_0}$ the complete Hopf fiber $\{(e^{i\theta}z,e^{i\theta}w):\theta\in\mathbb{R}\}$ stays in $T_{r_0}$, because multiplying both coordinates by $e^{i\theta}$ does not change their moduli. Lemma~\ref{lem:invariant-criterion} then shows that $T_{r_0}$ lies in no open hemisphere.
\end{proof}

\section{The rotational profile system}\label{sec:profile}

In this section, we prove the reduction of the CMC equation for rotational surfaces in the normalized Berger spheres to an explicit first-order profile system, and we derive the conserved quantity, the turning polynomial, and the holonomy and amplitude integrals.

Throughout this section we fix $\tau>0$ and work on $(\mathbb{S}^3,g_\tau)$ in the coordinate region $0<r<\tfrac{\pi}{2}$; recall the notation of Set-up~\ref{setup:berger}. A \emph{rotational surface} is an oriented surface invariant under the circle action $\Phi_\theta(z,w)=(z,e^{i\theta}w)$, generated by a regular profile curve $\gamma=(r,\alpha)$ in the orbit space.

\begin{prop}[The profile system]\label{prop:profile-system}
Let $\tau>0$, let $H\in\mathbb{R}$, and consider $(\mathbb{S}^3,g_\tau)$ as above. Then the following statements hold.
\begin{enumerate}
\item The orbit $\{\Phi_\theta(p)\}$ through a point $p$ with coordinates $(r,\alpha,\beta)$ has length $2\pi\ell$ with $\ell=s\sqrt{D}$. The metric $g_\tau$ is the ambient metric, the quotient metric on the orbit space is
\[
q=dr^2+\frac{\tau^2c^2}{D}\,d\alpha^2,
\]
and the area of a rotational surface equals $2\pi$ times the length of its profile in the weighted metric
\[
h=\ell^2 q=s^2D\,dr^2+\tau^2s^2c^2\,d\alpha^2 .
\]
\item A rotational surface has constant mean curvature $H$ with respect to the compatible unit normal, namely the unit normal whose horizontal projection is the $+\tfrac{\pi}{2}$-rotation of the profile tangent in the oriented orbit space, if and only if, after parametrizing its profile by $h$-arclength $u$ and defining the tangent angle $\theta\in\mathbb{R}/2\pi\mathbb{Z}$ by
\begin{equation}\label{eq:theta-def}
s\sqrt{D}\;\frac{dr}{du}=\cos\theta,\qquad \tau sc\;\frac{d\alpha}{du}=\sin\theta,
\end{equation}
the profile satisfies the autonomous system, where any local real lift of $\theta$ may be used to write $d\theta/du$,
\begin{equation}\label{eq:profile-system}
\frac{dr}{du}=\frac{\cos\theta}{s\sqrt{D}},\qquad
\frac{d\alpha}{du}=\frac{\sin\theta}{\tau sc},\qquad
\frac{d\theta}{du}=\frac{2\big(H-\cot(2r)\sin\theta\big)}{s\sqrt{D}} .
\end{equation}
\item Along every solution of \eqref{eq:profile-system}, the function
\begin{equation}\label{eq:energy}
\mathcal E(r,\theta)=sc\,\sin\theta-Hs^2
\end{equation}
is constant. Denote its scalar value along an orbit by $e$. Then $\sin\theta=(e+Ht)/\sqrt{t(1-t)}$, and consequently the motion is confined to the region $P(t)\geq 0$, where
\begin{equation}\label{eq:turning-polynomial}
P(t)=t(1-t)-(e+Ht)^2 .
\end{equation}
\item Let $H>0$. Prescribe a conserved value $e\in(-H,0)$ and set $a=-e/H\in(0,1)$; equivalently, choose $a\in(0,1)$ first and set $e=-aH$. This is the construction regime considered below. Then $P(t)=t(1-t)-H^2(t-a)^2$ has exactly two roots $t_1,t_2$ with $0<t_1<a<t_2<1$, both simple. Every solution of \eqref{eq:profile-system} on the level $\{\mathcal E=-aH\}$ is periodic in $(r,\theta)$: the coordinate $t=\sin^2 r$ oscillates between $t_1$ and $t_2$, and $d\theta/du\neq 0$ at both turning points. On a branch on which $t$ increases,
\begin{equation}\label{eq:dalphadt}
\frac{d\alpha}{dt}=\frac{H(t-a)\,R_\tau(t)}{2\tau\sqrt{P(t)}},
\qquad
R_\tau(t)=\frac{\sqrt{1-t+\tau^2t}}{\sqrt{t}\,(1-t)},
\end{equation}
and $d\alpha/dt$ changes sign exactly at $t=a$.
\item In the situation of $(4)$, the advance of $\alpha$ during one complete radial oscillation is
\begin{equation}\label{eq:holonomy}
\Delta\alpha=\frac{1}{\tau}\int_{t_1}^{t_2}\frac{H(t-a)\,R_\tau(t)}{\sqrt{P(t)}}\,dt,
\end{equation}
and the absolute change of $\alpha$ from the lower turning point to $t=a$ on such a branch is
\begin{equation}\label{eq:amplitude}
A=-\frac{1}{2\tau}\int_{t_1}^{a}\frac{H(t-a)\,R_\tau(t)}{\sqrt{P(t)}}\,dt\ >\ 0 .
\end{equation}
Both improper integrals converge, the endpoint singularities being of inverse square root type. The profile closes after finitely many radial oscillations if and only if $\Delta\alpha\in 2\pi\mathbb{Q}$, and its net winding around the axis $\{w=0\}$ vanishes if and only if $\Delta\alpha=0$.
\end{enumerate}
\end{prop}

\begin{proof}
We prove the proposition in five steps. We first compute the orbit length and the two profile metrics, and then derive the CMC system with an explicit hierarchy of the associated frames. We next isolate the planar dynamics and its first integral, analyze the selected energy level, and finally recover the angular motion and closing data.

\medskip

\noindent\textbf{Step 1.} We compute the orbit length and identify the two profile metrics.

Write $\eta_0=\tau^2-1$, so that by \eqref{eq:normalized-metric} the nonzero coefficients of the ambient metric $g_\tau$ in the coordinate order $(r,\alpha,\beta)$ are
\[
(g_\tau)_{rr}=1,\qquad
(g_\tau)_{\alpha\alpha}=c^2(1+\eta_0c^2),\qquad
(g_\tau)_{\alpha\beta}=\eta_0c^2s^2,\qquad
(g_\tau)_{\beta\beta}=s^2(1+\eta_0s^2)=s^2D .
\]
The determinant of the $\alpha\beta$-block is
\[
c^2s^2\big[(1+\eta_0c^2)(1+\eta_0s^2)-\eta_0^2c^2s^2\big]
=c^2s^2\big[1+\eta_0(c^2+s^2)\big]
=\tau^2c^2s^2 .
\]
The orbit of the action $\Phi_\theta$ is generated by the Killing field $K=\partial_\beta$, whence $\ell=|K|_{g_\tau}=s\sqrt{D}$ and the orbit length is $2\pi\ell$. Let $\pi$ denote the quotient map. Orthogonal projection away from $K$ gives
\[
q_{rr}=1,\qquad
q_{\alpha\alpha}=(g_\tau)_{\alpha\alpha}
-\frac{(g_\tau)_{\alpha\beta}^2}{(g_\tau)_{\beta\beta}}
=\frac{(g_\tau)_{\alpha\alpha}(g_\tau)_{\beta\beta}
-(g_\tau)_{\alpha\beta}^2}{(g_\tau)_{\beta\beta}}
=\frac{\tau^2c^2}{D},
\]
so $q$ is the quotient metric. The area of the rotational surface generated by a profile $\gamma$ equals
\[
\int 2\pi\ell\,\sqrt{q(\gamma',\gamma')}\,du
=2\pi\int\sqrt{h(\gamma',\gamma')}\,du,
\]
where $h=\ell^2q$ is the weighted metric governing profile length. This proves $(1)$.

\medskip

\noindent\textbf{Step 2.} We derive the CMC profile system and make the frame hierarchy explicit.

Let $(v,n)$ be the oriented $q$-orthonormal tangent and normal pair along the profile. Let $(V,N)$ be their horizontal lifts to $(\mathbb{S}^3,g_\tau)$, and let
\[
U=\frac{K}{\ell}
\]
be the unit orbit direction. Then
\[
d\pi(V)=v,\qquad d\pi(N)=n,\qquad d\pi(U)=0.
\]
Thus $(V,N,U)$ is a $g_\tau$-orthonormal ambient frame and
\[
T\Sigma=\operatorname{span}\{V,U\}.
\]

The horizontal part of $\nabla^{g_\tau}_VV$ projects to the $q$-covariant acceleration of the profile, so
\[
\mathrm{II}(V,V)=g_\tau(\nabla^{g_\tau}_VV,N)=k_q,
\]
where $k_q=q(\nabla^q_vv,n)$. For every horizontal vector $X$, the Killing equation gives
\[
g_\tau(\nabla^{g_\tau}_KK,X)
=-g_\tau(\nabla^{g_\tau}_XK,K)
=-\tfrac12 X\big(|K|_{g_\tau}^2\big).
\]
Dividing by $\ell^2$ yields $g_\tau(\nabla^{g_\tau}_UU,X)=-X(\log\ell)$. Hence $\mathrm{II}(U,U)=-N(\log\ell)$, and the trace identity reads
\begin{equation}\label{eq:trace-identity}
2H_{\Sigma}=k_q-N(\log\ell),
\end{equation}
where $H_\Sigma$ is the mean curvature function of the rotational surface.

For the weighted metric, set
\[
\widehat v=\ell^{-1}v,\qquad \widehat n=\ell^{-1}n.
\]
Then $(\widehat v,\widehat n)$ is an oriented $h$-orthonormal pair. Write $h=e^{2\varphi}q$ with $\varphi=\log\ell$. The conformal transformation rule
\[
\nabla^h_XY=\nabla^q_XY+X(\varphi)Y+Y(\varphi)X-q(X,Y)\,\nabla^q\varphi,
\]
gives
\[
\nabla^h_{\widehat v}\widehat v
=e^{-2\varphi}\big[\nabla^q_vv+v(\varphi)v-\nabla^q\varphi\big].
\]
Pairing with $\widehat n$ under $h=e^{2\varphi}q$ yields
\begin{equation}\label{eq:conformal-curvature}
k_h=h\big(\nabla^h_{\widehat v}\widehat v,\widehat n\big)
=\ell^{-1}\big[k_q-n(\log\ell)\big]
=\ell^{-1}\big[k_q-N(\log\ell)\big].
\end{equation}
Combining \eqref{eq:trace-identity} and \eqref{eq:conformal-curvature}, the rotational surface has constant mean curvature $H$ if and only if
\begin{equation}\label{eq:magnetic}
k_h=\frac{2H}{\ell}=\frac{2H}{s\sqrt{D}} .
\end{equation}

The normalized coordinate fields for $h$ are
\[
\widehat e_r=\frac{1}{s\sqrt D}\,\partial_r,
\qquad
\widehat e_\alpha=\frac{1}{\tau sc}\,\partial_\alpha.
\]
Write
\[
\widehat v=\cos\theta\,\widehat e_r+\sin\theta\,\widehat e_\alpha,
\qquad
\widehat n=-\sin\theta\,\widehat e_r+\cos\theta\,\widehat e_\alpha.
\]
This is precisely the definition \eqref{eq:theta-def}. The compatible ambient normal is the horizontal lift $N$ of $n$; replacing it by its negative replaces $H$ by $-H$ throughout.

Since the metric coefficients depend only on $r$, the Lie bracket is
\[
[\widehat e_r,\widehat e_\alpha]=-\mu\widehat e_\alpha,
\qquad
\mu=\frac{\cos(2r)}{s^2c\sqrt D}
=\frac{2\cot(2r)}{s\sqrt D},
\]
and the Koszul formula gives
\[
\nabla^h_{\widehat e_r}\widehat e_r
=\nabla^h_{\widehat e_r}\widehat e_\alpha=0,
\qquad
\nabla^h_{\widehat e_\alpha}\widehat e_\alpha=-\mu\widehat e_r,
\qquad
\nabla^h_{\widehat e_\alpha}\widehat e_r=\mu\widehat e_\alpha .
\]
Differentiating $\widehat v$ along the profile in $h$-arclength gives
\[
\nabla^h_{\widehat v}\widehat v
=\left(\frac{d\theta}{du}+\mu\sin\theta\right)\widehat n,
\]
and hence
\[
k_h=\frac{d\theta}{du}
+\frac{2\cot(2r)}{s\sqrt D}\,\sin\theta.
\]
Substituting into \eqref{eq:magnetic} gives the third equation of \eqref{eq:profile-system}, while the first two equations are \eqref{eq:theta-def}. This proves $(2)$.

\medskip

\noindent\textbf{Step 3.} We isolate the planar subsystem and obtain its first integral and turning polynomial.

The first and third equations of \eqref{eq:profile-system} form the planar autonomous subsystem
\begin{equation}\label{eq:planar-system}
\frac{dr}{du}=\frac{\cos\theta}{s\sqrt D},
\qquad
\frac{d\theta}{du}
=\frac{2\big(H-\cot(2r)\sin\theta\big)}{s\sqrt D}.
\end{equation}
The full profile system is a skew product over \eqref{eq:planar-system}: the equation for $\alpha$ depends on $(r,\theta)$ but not on $\alpha$, and $\alpha$ will be recovered by quadrature in \textbf{Step~5}.

Differentiating the function $\mathcal E$ of \eqref{eq:energy} along \eqref{eq:planar-system}, and writing $r'=\cos\theta/(s\sqrt D)$, gives
\[
\frac{d\mathcal E}{du}
=\cos(2r)\,r'\sin\theta+sc\cos\theta\,\frac{d\theta}{du}-2Hsc\,r'
\]
\[
=\frac{\cos(2r)\cos\theta\sin\theta}{s\sqrt D}
+\frac{2c\cos\theta\big(H-\cot(2r)\sin\theta\big)}{\sqrt D}
-\frac{2Hc\cos\theta}{\sqrt D}=0,
\]
because $2c\cot(2r)=\cos(2r)/s$. Thus $\mathcal E$ is a first integral. If its scalar value is $e$, solving \eqref{eq:energy} for $\sin\theta$ gives
\[
\sin\theta=\frac{e+Ht}{\sqrt{t(1-t)}}.
\]
Therefore $\sin^2\theta\leq1$ is equivalent to $P(t)\geq0$ for the turning polynomial \eqref{eq:turning-polynomial}. This proves $(3)$.

\medskip

\noindent\textbf{Step 4.} We choose the energy regime, locate the simple roots, and prove periodicity of the planar level.

Let $H>0$. We prescribe $e\in(-H,0)$ and set $a=-e/H\in(0,1)$, or equivalently choose $a\in(0,1)$ and set $e=-aH$. This is a choice of construction regime, not a conclusion extracted from the differential equation. On the level $\{\mathcal E=-aH\}$,
\[
P(t)=t(1-t)-H^2(t-a)^2
\]
satisfies
\[
P(0)=-a^2H^2<0,\qquad P(a)=a(1-a)>0,\qquad P(1)=-H^2(1-a)^2<0,
\]
so the downward quadratic $P$ has exactly two roots $t_1\in(0,a)$ and $t_2\in(a,1)$. These roots are simple because the two sign changes place distinct roots on opposite sides of $a$. Put
\[
r_j=\arcsin\sqrt{t_j}\qquad (j=1,2).
\]

An equilibrium of \eqref{eq:planar-system} satisfies $\cos\theta=0$ and $H=\sin\theta\cot(2r)$. At such a point $t$ is a root of $P$, and
\[
P'(t)=1-2t-2H^2(t-a)
=1-2t-2H\sin\theta\sqrt{t(1-t)}=0.
\]
Thus an equilibrium on $\{\mathcal E=-aH\}$ would give a repeated root, which has been excluded. Hence the level contains no equilibria.

The level is regular. We have
\[
\frac{\partial\mathcal E}{\partial\theta}=sc\cos\theta,
\qquad
\frac{\partial\mathcal E}{\partial r}
=\cos(2r)\sin\theta-2Hsc.
\]
A common zero would satisfy the equilibrium equations above, so $d\mathcal E$ does not vanish on the level. Moreover, the level lies in the compact cylinder
\[
[r_1,r_2]\times(\mathbb{R}/2\pi\mathbb{Z})
\subset(0,\tfrac{\pi}{2})\times(\mathbb{R}/2\pi\mathbb{Z}).
\]
For each $t\in(t_1,t_2)$, the energy relation determines exactly two points modulo $2\pi$: one with $\cos\theta>0$ and one with $\cos\theta<0$. These branches join at the two turning points and form one compact regular circle.

The radial speed satisfies
\begin{equation}\label{eq:radial-speed}
\left(\frac{dt}{du}\right)^2
=\frac{4P(t)}{t(1-t+\tau^2t)}.
\end{equation}
Thus $dt/du=0$ at $t=t_1,t_2$. At either turning point, however, the planar vector field does not vanish, so
\[
\frac{d\theta}{du}
=\frac{2\big(H-\sin\theta\cot(2r)\big)}{s\sqrt D}\neq0.
\]
This is the required turning-point transversality: the trajectory moves in the $\theta$-direction rather than crossing the turning value in the $t$-direction.

The planar vector field is nonzero on the compact regular circle, so every orbit traverses that circle in finite time and is periodic. Along one period, the lower branch has $t$ increasing from $t_1$ to $t_2$, while the upper branch has $t$ decreasing from $t_2$ to $t_1$. On the lines $\theta=0$ and $\theta=\pi$, the vector field in \eqref{eq:planar-system} is a positive multiple of $(1,2H)$ and $(-1,2H)$, respectively. Thus it points northeast along $\theta=0$ and northwest along $\theta=\pi$. These ambient directions, together with the selected periodic level, are shown schematically in Figure~\ref{fig:phase-cylinder}.

\begin{figure}[!htbp]
\centering
\begin{tikzpicture}[x=.12\linewidth,y=.9cm]
  \draw[gray] (0,0)--(0,4) (6,0)--(6,4);
  \draw[gray,dashed] (0,0)--(6,0) (0,4)--(6,4);
  \draw[gray,dotted] (0,2)--(6,2);

  % Normalized direction field for H=1/sqrt(3).
  % The plotting coordinates are x=12r/pi and y=1+2theta/pi.
  \begin{scope}[gray!55,line width=.3pt,->]
    \foreach \x in {.25,.75,1.25,1.75,2.25,2.75,3.25,3.75,4.25,4.75,5.25,5.75}{
      \foreach \y in {.3,1,1.5,1.85,2.15,2.5,3,3.7}{
        \pgfmathsetmacro{\thetafield}{90*(\y-1)}
        \pgfmathsetmacro{\vxfield}{3*cos(\thetafield)}
        \pgfmathsetmacro{\vyfield}
          {1/sqrt(3)-cos(30*\x)/sin(30*\x)*sin(\thetafield)}
        \pgfmathsetmacro{\normfield}
          {sqrt((1.93*\vxfield)*(1.93*\vxfield)
            +(\vyfield)*(\vyfield))}
        \pgfmathsetmacro{\dxfield}{.12*\vxfield/\normfield}
        \pgfmathsetmacro{\dyfield}{.12*\vyfield/\normfield}
        \draw
          ({\x-\dxfield},{\y-\dyfield})
          --
          ({\x+\dxfield},{\y+\dyfield});
      }
    }
  \end{scope}

  % The two equilibria; the points on the dashed edges are identified.
  \fill[gray!75] (2,2) circle[radius=1.2pt];
  \fill[gray!75] (4,0) circle[radius=1.2pt];
  \fill[gray!75] (4,4) circle[radius=1.2pt];

  \draw[gray] (1,0.08)--(1,-0.08) (5,0.08)--(5,-0.08);
  \node[below] at (1,-0.08) {$r_1$};
  \node[below] at (5,-0.08) {$r_2$};
  \node[below] at (3,-0.42) {$r$};
  \node[left] at (0,0) {$-\pi/2$};
  \node[left] at (0,1) {$0$};
  \node[left] at (0,2) {$\pi/2$};
  \node[left] at (0,3) {$\pi$};
  \node[left] at (0,4) {$3\pi/2$};
  \node[above] at (3,4) {$\theta\bmod 2\pi$};

  % The exact periodic level for H=1/sqrt(3) and a=1/2.
  \draw[black,line width=1pt,->]
    plot[domain=-90:90,samples=161,variable=\q]
      ({acos(-sqrt(3)*sin(\q)/sqrt(1+3*sin(\q)^2))/30},
       {1+\q/90});
  \draw[black,line width=1pt,->]
    plot[domain=90:270,samples=161,variable=\q]
      ({acos(-sqrt(3)*sin(\q)/sqrt(1+3*sin(\q)^2))/30},
       {1+\q/90});

  \node[fill=white,inner sep=1pt] at (3.6,.65) {$dt/du>0$};
  \node[fill=white,inner sep=1pt] at (3.4,3.35) {$dt/du<0$};
\end{tikzpicture}
\caption{A representative phase portrait of \eqref{eq:planar-system} for $H=1/\sqrt3$ and $a=1/2$. The light-gray arrows give the normalized direction field, while the thick black curve is the exact periodic level $\{\mathcal E=-1/(2\sqrt3)\}$ with $r_1=\pi/12$ and $r_2=5\pi/12$. The gray points mark the two equilibria, the dashed edges are identified, and the black arrows record the flow direction along the periodic orbit; arrow length does not encode speed.}
\label{fig:phase-cylinder}
\end{figure}
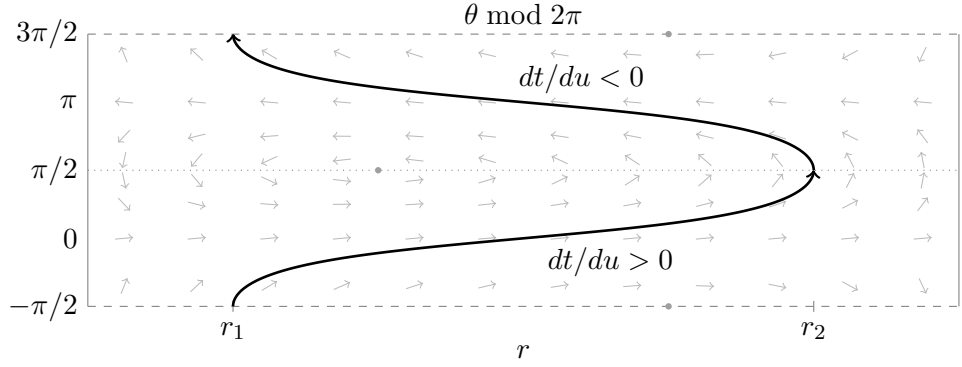

This proves the periodicity and the turning-point assertions of $(4)$.

\medskip

\noindent\textbf{Step 5.} We recover $\alpha$ by quadrature and derive holonomy, amplitude, closing, and winding.
\par\nopagebreak[4]

On the selected energy level,
\[
\frac{d\alpha}{du}
=\frac{\sin\theta}{\tau sc}
=\frac{H(t-a)}{\tau t(1-t)}.
\]
We treat the two radial branches separately.

\textbf{Case 1.} Suppose that $t$ increases from $t_1$ to $t_2$. Then
\[
\frac{dt}{du}
=\frac{2\sqrt{P(t)}}{\sqrt{t(1-t+\tau^2t)}},
\]
and division gives
\[
\frac{d\alpha}{dt}
=\frac{H(t-a)\sqrt{1-t+\tau^2t}}
{2\tau\sqrt t\,(1-t)\sqrt{P(t)}}
=\frac{H(t-a)R_\tau(t)}{2\tau\sqrt{P(t)}}.
\]
This is \eqref{eq:dalphadt}. All factors other than $t-a$ are positive on $(t_1,t_2)$, so $d\alpha/dt$ changes sign exactly once, at $t=a$. The contribution of this branch to the change of $\alpha$ is
\[
\int_{t_1}^{t_2}
\frac{H(t-a)R_\tau(t)}{2\tau\sqrt{P(t)}}\,dt.
\]

\textbf{Case 2.} Suppose that $t$ decreases from $t_2$ to $t_1$. The sign of $dt/du$ is reversed, while $d\alpha/du$ is unchanged as a function of $t$. Hence $d\alpha/dt$ is the negative of the expression in \textbf{Case~1}, and the contribution along the decreasing orientation is
\[
\int_{t_2}^{t_1}
-\frac{H(t-a)R_\tau(t)}{2\tau\sqrt{P(t)}}\,dt
=\int_{t_1}^{t_2}
\frac{H(t-a)R_\tau(t)}{2\tau\sqrt{P(t)}}\,dt.
\]
Adding the two contributions gives \eqref{eq:holonomy}. On the increasing branch the integrand is negative on $(t_1,a)$, so the absolute angular change from $t_1$ to $a$ is \eqref{eq:amplitude}, and it is positive.

The convergence now follows from the simplicity proved in \textbf{Step~4}. Near either root $t_j$, the function $P(t)$ vanishes to first order, while $R_\tau$ and the remaining factors are bounded because $0<t_1<t_2<1$. Thus the endpoint singularities are bounded by constant multiples of $|t-t_j|^{-1/2}$, and both improper integrals converge.

Finally, after each planar period the coordinate $\alpha$ changes by the same amount $\Delta\alpha$, because the full system is a skew product. Since $\alpha$ is taken modulo $2\pi$, the profile closes after finitely many periods if and only if $\Delta\alpha\in2\pi\mathbb{Q}$. If it closes after $m$ periods, its degree around the axis $\{w=0\}$ is $m\Delta\alpha/(2\pi)$; this degree vanishes if and only if $\Delta\alpha=0$. This proves the remaining assertions of $(4)$ and all assertions of $(5)$.
\end{proof}

\begin{rem}\label{rem:constant-r}
Constant-$r$ solutions of \eqref{eq:profile-system} satisfy $\cos\theta=0$ and $H=\sin\theta\cot(2r)$, and their rotational surfaces are the Hopf tori of Proposition~\ref{prop:hopf-tori} in the normalized family; the value $H=\cot(2r_0)$ there agrees with Proposition~\ref{prop:hopf-tori} at $\kappa=4$. They occur on energy levels other than the regime $H>0$, $e\in(-H,0)$ selected in Proposition~\ref{prop:profile-system}$(4)$. More explicitly, when $H>0$ and $\sin\theta=1$, their scalar energy is $e=\tfrac12\tan r>0$, while for $\sin\theta=-1$ it is $e=-\tfrac12\tan r<-H$. Thus no constant-$r$ solution lies in the energy band of $(4)$.
\end{rem}

\section{The lens construction}\label{sec:lens}

In this section, we prove that a zero-holonomy oscillation with small angular amplitude closes up after a single radial period, and that the resulting rotational surface is an embedded torus contained in an open hemisphere.

\begin{prop}[Lens lemma]\label{prop:lens}
Fix $\tau>0$, $H>0$, and $a\in(0,1)$. Let $t_1<t_2$ be the two roots of $P(t)=t(1-t)-H^2(t-a)^2$, and let $\Delta\alpha$ and $A$ be the holonomy \eqref{eq:holonomy} and the amplitude \eqref{eq:amplitude} of the energy level $\{\mathcal E=-aH\}$, as in Proposition~\ref{prop:profile-system}$(4)$--$(5)$. Let $\gamma$ be the solution of \eqref{eq:profile-system} with initial data
\[
r=\arcsin\sqrt{t_1},\qquad \alpha=0,\qquad \theta=-\tfrac{\pi}{2}.
\]
Assume that
\[
\Delta\alpha=0
\qquad\text{and}\qquad
A<\tfrac{\pi}{2}.
\]
Then the conserved function $\mathcal E$ of \eqref{eq:energy} has value $-aH$ along $\gamma$, the profile of $\gamma$ is a smooth simple closed curve in the orbit space, and its rotational surface $\Sigma$ is a smooth closed embedded torus of constant mean curvature $H$ in $(\mathbb{S}^3,g_\tau)$, contained in the open hemisphere $\mathcal{H}_{(1,0)}=\{\operatorname{Re}z>0\}$ with the quantitative margin
\[
x\cdot(1,0)\ \geq\ \sqrt{1-t_2}\,\cos A\ >\ 0
\qquad\text{for every }x\in\Sigma .
\]
\end{prop}

\begin{proof}
The solution $\gamma$ exists and is unique, since the coefficients of \eqref{eq:profile-system} are smooth on $0<r<\tfrac{\pi}{2}$. Its initial point lies on the energy level $\{\mathcal E=-aH\}$: at $t=t_1$ we have $P(t_1)=0$ and $t_1<a$, so $\sqrt{t_1(1-t_1)}=H(a-t_1)$, and
\[
sc\sin\theta-Hs^2=-\sqrt{t_1(1-t_1)}-Ht_1=-H(a-t_1)-Ht_1=-aH .
\]
By Proposition~\ref{prop:profile-system}$(4)$, the solution is periodic in $(r,\theta)$ and oscillates between $t_1$ and $t_2$.

Choose the parametrization so that $t$ increases immediately after the initial point, and let $\phi(t)$ denote the $\alpha$-coordinate along this increasing branch, so that by \eqref{eq:dalphadt},
\[
\phi(t)=\int_{t_1}^{t}\frac{H(x-a)\,R_\tau(x)}{2\tau\sqrt{P(x)}}\,dx,
\qquad t\in[t_1,t_2].
\]
The integrand is negative on $(t_1,a)$ and positive on $(a,t_2)$. Thus $\phi$ strictly decreases from $\phi(t_1)=0$ to $\phi(a)=-A$ and then strictly increases to $\phi(t_2)=\tfrac12\Delta\alpha=0$. In particular
\[
-A\leq\phi(t)<0\qquad\text{for all }t\in(t_1,t_2).
\]
On the subsequent decreasing branch, $d\alpha/dt$ is the negative of its value on the increasing branch, so starting from $\alpha=0$ at $t=t_2$ the $\alpha$-coordinate equals $-\phi(t)$, and the branch ends at $t=t_1$ with $\alpha=-\phi(t_1)=0$. The solution has then returned to its initial point in $(r,\alpha,\theta)$, so the profile is a closed curve traversed in one radial period, consisting of the two graphs
\[
\alpha=\phi(t)\in[-A,0]
\qquad\text{and}\qquad
\alpha=-\phi(t)\in[0,A],
\qquad t\in[t_1,t_2].
\]
The two graphs intersect only where $\phi(t)=-\phi(t)$, that is, only at the turning points $t\in\{t_1,t_2\}$, because $\phi<0$ on the open interval. All lifted $\alpha$-values lie in $[-A,A]$, an interval of length $2A<\pi<2\pi$, so reducing $\alpha$ modulo $2\pi$ creates no further identifications. The profile is smooth and regular: it is an orbit of the smooth system \eqref{eq:profile-system}, and at the turning points $d\alpha/du=\sin\theta/(\tau sc)=\mp 1/(\tau sc)\neq 0$ while elsewhere $dr/du\neq0$ or $d\alpha/du\neq 0$. Hence the profile is a smooth simple closed curve, a lens with vertices at the two turning points.

We now lift. Define
\[
F:\mathbb{S}^1\times(\mathbb{R}/2\pi\mathbb{Z})\longrightarrow\mathbb{S}^3,
\qquad
F(u,\beta)=\big(\cos r(u)\,e^{i\alpha(u)},\,\sin r(u)\,e^{i\beta}\big),
\]
where $u$ runs over one period of the profile. Since $0<t_1\leq t(u)\leq t_2<1$, both complex coordinates have nonzero moduli. Suppose $F(u,\beta)=F(u',\beta')$. Comparing moduli gives $t(u)=t(u')$; comparing arguments of the first coordinates gives $\alpha(u)=\alpha(u')$ modulo $2\pi$, hence $\alpha(u)=\alpha(u')$ because the lifted $\alpha$-range has length less than $2\pi$; simplicity of the profile then forces $u=u'$, and comparing arguments of the second coordinates gives $\beta=\beta'$. Thus $F$ is injective. It is an immersion, because the profile velocity has a nonzero $(r,\alpha)$-component while the orbit direction is purely in $\beta$, and these are linearly independent. A smooth injective immersion of a compact surface is an embedding, so $\Sigma=F(\mathbb{S}^1\times\mathbb{S}^1)$ is a smooth closed embedded torus, and it has constant mean curvature $H$ by Proposition~\ref{prop:profile-system}$(2)$.

Finally, every point of $\Sigma$ has the form $x=(\sqrt{1-t}\,e^{i\alpha},\sqrt{t}\,e^{i\beta})$ with $t\leq t_2$ and $|\alpha|\leq A<\tfrac{\pi}{2}$. Lemma~\ref{lem:hemisphere-criterion} with $\alpha_0=0$, $\delta=A$, and $t_+=t_2$ gives $x\cdot(1,0)\geq\sqrt{1-t_2}\cos A>0$.
\end{proof}

\begin{rem}\label{rem:anchoring}
The normalization of the lower turning point at $\alpha=0$ in Proposition~\ref{prop:lens} is needed for the stated pole $(1,0)$: the isometries $(z,w)\mapsto(e^{i\sigma}z,w)$ of $g_\tau$ translate profiles in $\alpha$ and preserve all the other hypotheses, while rotating the hemisphere to $\mathcal{H}_{(e^{i\sigma},0)}$.
\end{rem}

\section{Zero holonomy at large mean curvature}\label{sec:asymptotic}

In this section, we prove the holonomy asymptotics and produce zero-holonomy parameters for all sufficiently large mean curvature.

Throughout this section we fix $\tau\in(0,\tfrac{1}{\sqrt2})$ and use the notation
\[
f(t)=t(1-t),\qquad
D_\tau(t)=1-t+\tau^2t,\qquad
R_\tau(t)=\frac{\sqrt{D_\tau(t)}}{\sqrt{t}\,(1-t)},\qquad
P_{H,a}(t)=f(t)-H^2(t-a)^2 ,
\]
and we write $\Delta_{H,\tau}(a)$ and $A_{H,\tau}(a)$ for the holonomy \eqref{eq:holonomy} and the amplitude \eqref{eq:amplitude} of the energy level $\{\mathcal E=-aH\}$, which are defined for all $H>0$ and $a\in(0,1)$ by Proposition~\ref{prop:profile-system}$(4)$--$(5)$.

\begin{prop}[Holonomy asymptotics]\label{prop:asymptotic}
Fix $\tau\in(0,\tfrac{1}{\sqrt2})$ and set
\[
a_c=\frac{1}{2(1-\tau^2)}\in\big(\tfrac12,1\big).
\]
Then, uniformly for $a$ in every compact subinterval of $(0,1)$,
\begin{equation}\label{eq:limit}
\lim_{H\to\infty}H^2\,\Delta_{H,\tau}(a)
=\frac{\pi}{2\tau}\,\frac{d}{da}\sqrt{a\,D_\tau(a)}
=\frac{\pi\,\big[1+2(\tau^2-1)a\big]}{4\tau\sqrt{a\,D_\tau(a)}},
\end{equation}
and $A_{H,\tau}(a)=O(1/H)$ uniformly on every such compact interval. The right-hand side of \eqref{eq:limit} is positive for $a<a_c$ and negative for $a>a_c$. Consequently, for all $a_-,a_+$ with $\tfrac12<a_-<a_c<a_+<1$ there exists $H_0>0$ such that for every $H\geq H_0$ there exists $a_H\in(a_-,a_+)$ with
\[
\Delta_{H,\tau}(a_H)=0,
\qquad\text{and}\qquad
A_{H,\tau}(a_H)\longrightarrow 0\ \text{ as }H\to\infty .
\]
\end{prop}

\begin{proof}
Put $\varepsilon=1/H$ and fix a compact subinterval $K\subset(0,1)$. We prove the proposition in three steps.

\medskip

\noindent\textbf{Step 1.} In this step, we regularize the holonomy integral by an exact change of variables.

Substituting $t=a+\varepsilon x$ into the quadratic $f$ gives the exact identity
\[
P_{H,a}(a+\varepsilon x)=f(a)+\varepsilon f'(a)x-(1+\varepsilon^2)x^2 .
\]
Write $B=1+\varepsilon^2$ and
\[
x_c=\frac{\varepsilon f'(a)}{2B},
\qquad
L^2=\frac{f(a)}{B}+\frac{\varepsilon^2f'(a)^2}{4B^2},
\]
so that completing the square yields $P_{H,a}(a+\varepsilon x)=B\big[L^2-(x-x_c)^2\big]$, with roots at $x=x_c\pm L$. On $K$, $L\to\sqrt{f(a)}>0$ and $x_c\to 0$ uniformly as $\varepsilon\to0$, so $t_{1,2}=a+\varepsilon(x_c\mp L)$, consistent with Proposition~\ref{prop:profile-system}$(4)$.

Using the fixed-endpoint substitution $x=x_c+Ly$ with $y\in[-1,1]$, so that $dt=\varepsilon L\,dy$ and $\sqrt{P_{H,a}}=\sqrt{B}\,L\sqrt{1-y^2}$, the holonomy \eqref{eq:holonomy} becomes the exact regularized integral
\begin{equation}\label{eq:regularized}
\Delta_{H,\tau}(a)=\frac{\varepsilon}{\tau\sqrt{B}}\int_{-1}^{1}
\frac{\big(x_c+Ly\big)\,R_\tau\big(a+\varepsilon(x_c+Ly)\big)}{\sqrt{1-y^2}}\;dy .
\end{equation}

\medskip

\noindent\textbf{Step 2.} In this step, we prove the limit \eqref{eq:limit} and the amplitude bound.

All arguments of $R_\tau$ in \eqref{eq:regularized} remain in a fixed compact subinterval of $(0,1)$, uniformly for $a\in K$ and small $\varepsilon$, so $R_\tau$ and its first two derivatives are uniformly bounded there. Taylor's theorem gives, uniformly in $a$ and $y$,
\[
R_\tau(a+\varepsilon x)=R_\tau(a)+\varepsilon x\,R_\tau'(a)+\varepsilon^2x^2E(a,\varepsilon,x)
\]
with $E$ uniformly bounded, and $x=x_c+Ly$ is uniformly bounded. Using
\[
\int_{-1}^{1}\frac{dy}{\sqrt{1-y^2}}=\pi,\qquad
\int_{-1}^{1}\frac{y\,dy}{\sqrt{1-y^2}}=0,\qquad
\int_{-1}^{1}\frac{y^2\,dy}{\sqrt{1-y^2}}=\frac{\pi}{2},
\]
we obtain
\[
\int_{-1}^{1}\frac{x\,R_\tau(a+\varepsilon x)}{\sqrt{1-y^2}}\,dy
=\pi x_cR_\tau(a)+\varepsilon\pi\Big[x_c^2+\frac{L^2}{2}\Big]R_\tau'(a)+O(\varepsilon^2),
\]
uniformly on $K$. Substituting $x_c=\varepsilon f'(a)/(2B)=O(\varepsilon)$, $x_c^2=O(\varepsilon^2)$, and $L^2=f(a)/B+O(\varepsilon^2)$ into \eqref{eq:regularized} gives
\[
\Delta_{H,\tau}(a)=\frac{\varepsilon^2\pi}{2\tau}\Big[f'(a)R_\tau(a)+f(a)R_\tau'(a)\Big]+O(\varepsilon^3),
\]
uniformly on $K$. Since
\[
f(a)R_\tau(a)=\frac{a(1-a)\sqrt{D_\tau(a)}}{\sqrt{a}\,(1-a)}=\sqrt{a\,D_\tau(a)},
\]
the bracket equals $\frac{d}{da}\sqrt{aD_\tau(a)}$, which proves \eqref{eq:limit}; differentiating $aD_\tau(a)=a+(\tau^2-1)a^2$ gives the displayed quotient.

For the amplitude, the same substitution applies to \eqref{eq:amplitude}, with the upper endpoint $t=a$ corresponding to $y_0=-x_c/L$:
\[
A_{H,\tau}(a)=\frac{\varepsilon}{2\tau\sqrt{B}}\int_{-1}^{y_0}
\frac{-\big(x_c+Ly\big)\,R_\tau\big(a+\varepsilon(x_c+Ly)\big)}{\sqrt{1-y^2}}\;dy .
\]
The factors $x_c+Ly$ and $R_\tau$ are uniformly bounded on $K$, and $\int_{-1}^{1}(1-y^2)^{-1/2}dy=\pi$, so $0\leq A_{H,\tau}(a)\leq C_K/H$ for a constant $C_K$ independent of $H$ and $a$.

\medskip

\noindent\textbf{Step 3.} We conclude the proof in this step.

Since $0<\tau^2<\tfrac12$, we have $1-\tau^2\in(\tfrac12,1)$ and hence $a_c=\frac{1}{2(1-\tau^2)}\in(\tfrac12,1)$. The sign of the limit function in \eqref{eq:limit} is the sign of the affine function $1-2(1-\tau^2)a$, which is positive for $a<a_c$ and negative for $a>a_c$. Fix $\tfrac12<a_-<a_c<a_+<1$ and take $K=[a_-,a_+]$. Uniform convergence on $K$ gives $H_0$ such that $\Delta_{H,\tau}(a_-)>0$ and $\Delta_{H,\tau}(a_+)<0$ for every $H\geq H_0$. Formula \eqref{eq:regularized} has fixed endpoints and an integrable dominating multiple of $(1-y^2)^{-1/2}$, so $a\mapsto\Delta_{H,\tau}(a)$ is continuous on $K$ by dominated convergence. The intermediate value theorem supplies $a_H\in(a_-,a_+)$ with $\Delta_{H,\tau}(a_H)=0$, and by \textbf{Step~2}, $A_{H,\tau}(a_H)\leq C_K/H\to 0$.
\end{proof}

\section{Proof of Theorem~\ref{thm:main-intro}}\label{sec:main-proof}

In this section, we combine the profile system, the lens lemma, and the holonomy asymptotics and prove Theorem~\ref{thm:main-intro}.

\begin{thm}[Normalized main theorem]\label{thm:main-normalized}
Let $\tau\in(0,\tfrac{1}{\sqrt2})$. Then there exists $H_0>0$ such that for every $H\geq H_0$, the normalized Berger sphere $(\mathbb{S}^3,g_\tau)$ contains a smooth closed embedded rotational torus of constant mean curvature $H$ contained in the open hemisphere $\{\operatorname{Re}z>0\}$.
\end{thm}

\begin{proof}
Fix $a_\pm$ with $\tfrac12<a_-<a_c<a_+<1$, where $a_c=\frac{1}{2(1-\tau^2)}$ as in Proposition~\ref{prop:asymptotic}. Let $H_0$ be as in Proposition~\ref{prop:asymptotic}, enlarged so that $A_{H,\tau}(a_H)<\tfrac{\pi}{2}$ for all $H\geq H_0$; this is possible by the uniform bound $A_{H,\tau}(a_H)\leq C_K/H$.

Fix $H\geq H_0$ and put $a=a_H$. The holonomy and amplitude of the energy level $\{\mathcal E=-aH\}$ are $\Delta_{H,\tau}(a_H)=0$ and $A_{H,\tau}(a_H)<\tfrac{\pi}{2}$, so Proposition~\ref{prop:lens}, applied with this $a$, produces a smooth closed embedded rotational torus of constant mean curvature $H$ in $(\mathbb{S}^3,g_\tau)$ contained in $\{\operatorname{Re}z>0\}$.
\end{proof}

\begin{proof}[Proof of Theorem~\ref{thm:main-intro}]
Since the metric \eqref{eq:berger-metric} depends only on $\tau^2$, we may assume $\tau>0$. Put $\lambda=2/\sqrt{\kappa}$ and $\tau'=\lambda\tau$. By hypothesis $\kappa>8\tau^2$, so $\tau'\in(0,\tfrac{1}{\sqrt2})$ by Lemma~\ref{lem:scaling}. Let $H_0=H_0(\tau')$ be as in Theorem~\ref{thm:main-normalized} and set $H_*=H_0/\lambda=\tfrac{\sqrt{\kappa}}{2}H_0$.

Let $H\geq H_*$ and put $H'=\lambda H\geq H_0$. Theorem~\ref{thm:main-normalized} provides a smooth closed embedded rotational torus $\Sigma\subset\mathbb{S}^3$ of constant mean curvature $H'$ with respect to $g_{\tau'}$, contained in $\{\operatorname{Re}z>0\}$. By Lemma~\ref{lem:scaling}, $g_{\tau'}=g_{4,\tau'}=\lambda^{-2}g_{\kappa,\tau}$, and the same oriented surface $\Sigma$ has constant mean curvature $H'/\lambda=H$ with respect to $g_{\kappa,\tau}$. Rescaling the ambient metric by a constant does not change the underlying subset of $\mathbb{S}^3$, its smoothness, compactness, embeddedness, genus, or rotational invariance, and the hemisphere condition of Definition~\ref{defn:hemisphere} is metric-independent. Thus $\Sigma$ is a smooth closed embedded rotational torus of constant mean curvature $H$ in $\mathbb{E}^3(\kappa,\tau)$ contained in an open hemisphere.

For the final assertion, $\Sigma$ has genus one, so it is not a sphere; in particular it is a compact embedded CMC surface in an open hemisphere of $\mathbb{E}^3(\kappa,\tau)$ which is not a canonical rotational CMC sphere.
\end{proof}

\section{An explicit torus: \texorpdfstring{$\tau=\frac12$, $H=20$}{tau=1/2, H=20}}\label{sec:explicit}

In this section, we prove Theorem~\ref{thm:explicit-intro} by certifying one zero-holonomy parameter at $\tau=\tfrac12$ and $H=20$ with exact rational arithmetic.

Throughout this section we fix
\[
\tau=\tfrac12,\qquad H=20,
\]
so that $R(t):=R_{1/2}(t)=\dfrac{\sqrt{1-\tfrac34 t}}{\sqrt{t}\,(1-t)}$, and for $a\in[\tfrac12,\tfrac45]$ we write
\[
P_a(t)=t(1-t)-400(t-a)^2,
\]
\[
\Delta(a)=2\int_{t_1(a)}^{t_2(a)}\frac{20(t-a)R(t)}{\sqrt{P_a(t)}}\,dt,
\qquad
A(a)=-\int_{t_1(a)}^{a}\frac{20(t-a)R(t)}{\sqrt{P_a(t)}}\,dt,
\]
where $t_1(a)<t_2(a)$ are the roots of $P_a$. These agree with \eqref{eq:holonomy} and \eqref{eq:amplitude}, since $1/\tau=2$ and $1/(2\tau)=1$.

\begin{thm}[Certified zero holonomy]\label{thm:certificate}
There exists $a_*\in(\tfrac12,\tfrac45)$ such that
\[
\Delta(a_*)=0
\qquad\text{and}\qquad
A(a_*)<\frac{7000}{20451}\,\pi<\frac{\pi}{2}.
\]
\end{thm}

\begin{proof}
We prove the theorem in five steps.

\medskip

\noindent\textbf{Step 1.} In this step, we locate the root band.

Expanding gives $P_a(t)=-401t^2+(1+800a)t-400a^2$, with discriminant $1+1600a(1-a)>0$ on $[\tfrac12,\tfrac45]$. The two roots are $c(a)\pm d(a)$ with
\[
c(a)=\frac{1+800a}{802},\qquad d(a)=\frac{\sqrt{1+1600a(1-a)}}{802}.
\]
For $a\in[\tfrac12,\tfrac45]$ we have $c(a)\geq\tfrac12$, $c(a)<\tfrac45$, and $d(a)<\tfrac{21}{802}$, because $1+1600a(1-a)\leq 401<21^2$. Moreover $\tfrac12-\tfrac{21}{802}=\tfrac{190}{401}>\tfrac{9}{20}$ and $\tfrac45+\tfrac{21}{802}<\tfrac45+\tfrac{3}{100}=\tfrac{83}{100}$, so
\begin{equation}\label{eq:root-band}
\tfrac{9}{20}<c(a)-d(a)<c(a)+d(a)<\tfrac{83}{100}
\qquad\text{for all }a\in[\tfrac12,\tfrac45].
\end{equation}
All denominators appearing below are therefore uniformly bounded away from zero.

\medskip

\noindent\textbf{Step 2.} In this step, we remove the endpoint singularities and reduce to a continuous integral.

Substituting $t=c(a)-d(a)\cos\vartheta$ with $\vartheta\in[0,\pi]$ and using $P_a(t)=401\,(t-t_1)(t_2-t)$, the square root endpoint factors cancel, and
\[
\Delta(a)=\frac{40}{\sqrt{401}}\,J(a),
\qquad
J(a)=\int_0^\pi\big(t-a\big)\,W(t)\,d\vartheta,
\qquad
W(t)=\frac{\sqrt{1-\tfrac34t}}{\sqrt{t}\,(1-t)} .
\]
By \eqref{eq:root-band}, the integrand of $J$ is jointly continuous on the compact rectangle $[\tfrac12,\tfrac45]\times[0,\pi]$, so $J$ is continuous on $[\tfrac12,\tfrac45]$, and $\Delta(a)$ and $J(a)$ have the same sign.

\medskip

\noindent\textbf{Step 3.} In this step, we prove $J(\tfrac12)>0$.

At $a=\tfrac12$ we have $c(a)=a$. Pairing $\vartheta$ with $\pi-\vartheta$ over $[0,\tfrac{\pi}{2}]$, the paired contribution at $x=d\cos\vartheta>0$ is
\[
x\big[W(a+x)-W(a-x)\big],
\]
which is positive provided $W$ is strictly increasing on the root interval. Logarithmic differentiation gives
\[
\frac{W'(t)}{W(t)}=\frac{1}{1-t}-\frac{1}{2t}-\frac{3}{8\big(1-\tfrac34t\big)} ,
\]
and after multiplication by the positive quantity $8t(1-t)(1-\tfrac34t)$ the numerator becomes $2(-2+6t-3t^2)$. The quadratic $-2+6t-3t^2$ is increasing on $t<1$ and its value at $t=\tfrac{9}{20}$ is $\tfrac{37}{400}>0$, so it is positive on the whole band \eqref{eq:root-band}. Thus $W$ is strictly increasing there, every paired contribution except the midpoint is positive, and $J(\tfrac12)>0$.

\medskip

\noindent\textbf{Step 4.} In this step, we prove $J(\tfrac45)<0$ by exact rational estimates.

Set $a=\tfrac45$, so that
\[
c=\frac{641}{802},\qquad d=\frac{\sqrt{257}}{802},\qquad \delta:=c-a=-\frac{3}{4010}.
\]
Since $d<\tfrac{17}{802}$, the whole $t$-interval lies in $(\tfrac34,\tfrac{83}{100})$: we have $c-d>\tfrac{624}{802}>\tfrac34$ and $c+d<\tfrac{658}{802}<\tfrac{83}{100}$.

Write $L=\log W$. Then
\begin{gather*}
L'=\frac{1}{1-t}-\frac{1}{2t}-\frac{3}{8(1-\tfrac34t)},\qquad
L''=\frac{1}{(1-t)^2}+\frac{1}{2t^2}-\frac{9}{32(1-\tfrac34t)^2},\\
L'''=\frac{2}{(1-t)^3}-\frac{1}{t^3}-\frac{27}{64(1-\tfrac34t)^3}.
\end{gather*}
On $(\tfrac34,\tfrac{83}{100})$ we have $t>\tfrac34$, $\sqrt{t}>\tfrac45$, $1-t>\tfrac{17}{100}$, and $1-\tfrac34t>\tfrac38$, so $W<\tfrac{125}{17}<\tfrac{15}{2}$ and
\[
|L'|<8,\qquad |L''|<38,\qquad |L'''|<418 .
\]
From $W'''=W\big[(L')^3+3L'L''+L'''\big]$ we conclude
\[
|W'''|<\tfrac{15}{2}\,(512+912+418)=13815
\]
on the whole interval.

At the center $c$ we have $1-\tfrac34c=\tfrac{1285}{3208}$ and $1-c=\tfrac{161}{802}$. The square comparisons
\[
\Big(\tfrac{1581}{2500}\Big)^2<\tfrac{1285}{3208},
\qquad
c<\Big(\tfrac{9}{10}\Big)^2
\]
give $\sqrt{1-\tfrac34c}>\tfrac{1581}{2500}$ and $\sqrt{c}<\tfrac{9}{10}$, hence
\[
W(c)>\frac{1581/2500}{(9/10)(161/802)}=\frac{12679620}{3622500}>\frac72,
\]
the last inequality because $2\cdot 12679620=25359240>25357500=7\cdot 3622500$. In the same way,
\[
\tfrac{1285}{3208}<\Big(\tfrac{127}{200}\Big)^2,
\qquad
\Big(\tfrac{89}{100}\Big)^2<c
\]
give
\[
W(c)<\frac{127/200}{(89/100)(161/802)}=\frac{10185400}{2865800}<\frac{18}{5},
\]
because $5\cdot 10185400=50927000<51584400=18\cdot 2865800$. The exact rational value
\[
L'(c)=\frac{453483682}{132613285}<\frac{185}{54}
\]
gives $W'(c)=W(c)L'(c)<\tfrac{18}{5}\cdot\tfrac{185}{54}=\tfrac{37}{3}$. The second derivative $W''(c)=W(c)\big[(L'(c))^2+L''(c)\big]$ is positive, because
\[
L''(c)>\frac{1}{(1-c)^2}-\frac{9}{32(1-\tfrac34c)^2}>16-\frac{225}{128}>0 .
\]

For $y=-d\cos\vartheta$, Taylor's theorem at $c$ gives
\[
W(c+y)=W(c)+W'(c)\,y+\tfrac12W''(c)\,y^2+E(\vartheta),
\qquad
|E(\vartheta)|\leq\frac{13815\,|y|^3}{6} .
\]
Using
\[
\int_0^\pi y\,d\vartheta=0,\quad
\int_0^\pi y^2\,d\vartheta=\frac{\pi d^2}{2},\quad
\int_0^\pi y^3\,d\vartheta=0,\quad
\int_0^\pi |y|^3\,d\vartheta=\frac{4d^3}{3},\quad
\int_0^\pi y^4\,d\vartheta=\frac{3\pi d^4}{8},
\]
and $t-a=\delta+y$, we obtain
\[
J(\tfrac45)\ \leq\ \pi\Big[\delta\,W(c)+\frac{d^2}{2}W'(c)+\frac{\delta d^2}{4}W''(c)\Big]
+\frac{13815}{6}\Big[\frac{4}{3}|\delta|\,d^3+\frac{3\pi}{8}d^4\Big].
\]
The term $\frac{\delta d^2}{4}W''(c)$ is negative, because $\delta<0$ and $W''(c)>0$, and we discard it. Since $d^2=\tfrac{257}{643204}$, the remaining bracket satisfies
\[
\delta\,W(c)+\frac{d^2}{2}W'(c)
<-\frac{21}{8020}+\frac{257}{2\cdot 643204}\cdot\frac{37}{3}
=-\frac{2981}{19296120}<0 .
\]
Using $\pi>3$ in this negative term, $\pi<\tfrac{22}{7}$ in the positive remainder, and $d<\tfrac{17}{802}$, exact rational arithmetic gives
\begin{multline*}
J(\tfrac45)
<3\cdot\Big(-\frac{2981}{19296120}\Big)
+\frac{13815}{6}\left[\frac{4}{3}\cdot\frac{3}{4010}\cdot\Big(\frac{17}{802}\Big)^3
+\frac{3}{8}\cdot\frac{22}{7}\cdot\Big(\frac{257}{643204}\Big)^2\right]\\
=-\frac{967341167}{115839187972480}<0 .
\end{multline*}

By \textbf{Steps~2--4} and the intermediate value theorem, there exists $a_*\in(\tfrac12,\tfrac45)$ with $J(a_*)=0$, hence $\Delta(a_*)=0$.

\medskip

\noindent\textbf{Step 5.} In this step, we bound the amplitude and conclude.

The substitution of \textbf{Step~2}, stopped at the value $\vartheta_*\in(0,\pi)$ where $t=a_*$, gives
\[
A(a_*)=\frac{20}{\sqrt{401}}\int_0^{\vartheta_*}\big(a_*-t\big)W(t)\,d\vartheta .
\]
Here $\tfrac{20}{\sqrt{401}}<1$ and $0\leq a_*-t\leq t_2-t_1=\frac{\sqrt{1+1600a_*(1-a_*)}}{401}<\frac{21}{401}$ on $[0,\vartheta_*]$. On the band \eqref{eq:root-band} the function $\sqrt{t}(1-t)$ is decreasing, because $t>\tfrac13$, so
\[
\sqrt{t}\,(1-t)>\sqrt{\tfrac{83}{100}}\cdot\tfrac{17}{100}>\tfrac{9}{10}\cdot\tfrac{17}{100}=\tfrac{153}{1000},
\]
while $\sqrt{1-\tfrac34t}<1$; hence $W(t)<\tfrac{1000}{153}$ on the band. Therefore
\[
A(a_*)<\pi\cdot\frac{21}{401}\cdot\frac{1000}{153}=\frac{7000}{20451}\,\pi<\frac{\pi}{2},
\]
because $2\cdot 7000=14000<20451$.
\end{proof}

\begin{proof}[Proof of Theorem~\ref{thm:explicit-intro}]
The displayed metric is $g_{1/2}$ in the notation \eqref{eq:normalized-metric}, since $\tau^2-1=-\tfrac34$ for $\tau=\tfrac12$; by Lemma~\ref{lem:scaling} this is the Berger sphere $\mathbb{E}^3(4,\tfrac12)$, and it is non-round since $4\neq 4\tau^2=1$.

Let $a_*\in(\tfrac12,\tfrac45)$ be as in Theorem~\ref{thm:certificate}. The holonomy and amplitude of the energy level $\{\mathcal E=-20a_*\}$ are the quantities $\Delta(a_*)=0$ and $A(a_*)<\tfrac{\pi}{2}$ of Theorem~\ref{thm:certificate}. Proposition~\ref{prop:lens}, applied with $\tau=\tfrac12$, $H=20$, and $a=a_*$, yields a smooth closed embedded rotational torus of constant mean curvature $20$ in $\mathbb{E}^3(4,\tfrac12)$, contained in the open hemisphere $\{\operatorname{Re}z>0\}$, with margin $\operatorname{Re}z\geq\sqrt{1-t_2}\cos A(a_*)>0$ at every point.
\end{proof}

\section{Concluding remarks}\label{sec:remarks}

The results proved above settle the Berger case of the conjecture of \cite[Section~4.5]{FM10} in the range $\kappa>8\tau^2$. In this final section, we collect three remarks on the relation of the construction to the existing literature.

\begin{rem}[Nodoids can close up embeddedly]\label{rem:nodoid-correction}
Along the tori of Theorems~\ref{thm:main-normalized} and~\ref{thm:certificate}, the conserved quantity is $e=-aH\in(-H,0)$, so the angular velocity $d\alpha/dt$ changes sign along the profile by Proposition~\ref{prop:profile-system}$(4)$. In the classifications of rotational CMC and prescribed mean curvature surfaces in Berger spheres, this sign change is the defining feature of the \emph{nodoid} case; see \cite[Section~3]{Tor10} and \cite[Theorem~1.2]{Bue22}. In the normalization $\kappa=4$, the period $T(H,E)$ of \cite{Tor10} equals our $\Delta\alpha$, with $E=e$, so $T=0$ is formally included in the compactness condition that $T$ be a rational multiple of $\pi$. The embedded compact examples identified in \cite[Section~3]{Tor10} arise in the unduloid case with $T=2\pi/k$ for a positive integer $k$; its nodoid discussion says only that the resulting surface is immersed and does not identify an embedded zero-period subcase. Similarly, item~(4) of \cite[Theorem~1.2]{Bue22} describes rotational nodoids as properly immersed with self-intersections, while the embeddedness criterion of \cite[Section~4.1]{Bue22} concerns the nonzero vertical periods $T=8\tau\pi/(p\kappa)$ with $p$ a nonzero integer. The tori constructed here have $T=\Delta\alpha=0$; at this parameter the profile is a simple lens rather than a chain of loops, and embeddedness follows from Proposition~\ref{prop:lens}. Thus the zero-period case is an embedded exception to the usual self-intersecting nodoid picture.
\end{rem}

\begin{rem}[Isoperimetry and stability]\label{rem:isoperimetric}
The isoperimetric problem in the Berger spheres $\mathbb{E}^3(\kappa,\tau)$ with $\tfrac13\leq\tfrac{4\tau^2}{\kappa}<1$ is solved by Torralbo and Urbano: the solutions are the canonical rotational CMC spheres \cite{TU12}; see also \cite[Section~4.4]{FM10}. The parameter ranges overlap: for $\tfrac13\leq\tfrac{4\tau^2}{\kappa}<\tfrac12$, Theorem~\ref{thm:main-intro} produces embedded CMC tori in hemispheres of Berger spheres in which the isoperimetric regions are all bounded by spheres. The two statements do not conflict: the tori of Theorem~\ref{thm:main-intro} do not bound isoperimetric regions.
\end{rem}

\begin{rem}[The remaining spaces]\label{rem:remaining}
The construction of this paper is rotational, and its reach is limited by the geometry of the rotational families. In $\mathrm{Nil}_3$ and in the universal cover of $\mathrm{PSL}_2(\mathbb{R})$, compact rotational CMC tori do not exist, because rotational CMC tori in $\mathbb{E}^3(\kappa,\tau)$ exist only when $\kappa-4\tau^2>0$ (see \cite[Section~4.1]{FM10}); the corresponding cases of the conjecture of \cite[Section~4.5]{FM10}, as well as the Berger ranges not covered by Theorem~\ref{thm:main-intro}, remain open; see Question~\ref{ques:all-berger}.
\end{rem}

\end{document}